\documentclass[12pt,twoside]{amsart}
\usepackage{amsmath, amsthm, amscd, amsfonts, amssymb, graphicx}
\usepackage{enumerate}
\usepackage[colorlinks=true,
linkcolor=blue,
urlcolor=cyan,
citecolor=red]{hyperref}
\usepackage{mathrsfs}
\newtheorem{theorem}{Theorem}[section]
\newtheorem{lemma}{Lemma}[section]
\newtheorem{remark}{Remark}[section]

\newtheorem{corollary}{Corollary}[section]

\numberwithin{equation}{section}

\begin{document}
	
\title{ Norm and numerical radius inequalities for operator matrices}
\author{Fuad Kittaneh, Hamid Reza Moradi and Mohammad Sababheh}

\begin{abstract}
Operator matrices have played a significant role in studying Hilbert space operators. In this paper, we discuss further properties of operator matrices and present new estimates for the operator norms and numerical radii of such operators. Moreover, operator matrices whose real and imaginary parts are positive will be discussed, and sharper bounds will be shown for such class.
\end{abstract}
\maketitle
\pagestyle{myheadings}
\markboth{\centerline {}}
{\centerline {}}
\bigskip
\bigskip
\section{Introduction and Preliminaries}
Let $\mathcal{H}$ be a complex Hilbert space and let $\mathcal{B}(\mathcal{H})$ be the $C^*-$algebra of all bounded linear operators on $\mathcal{H}$. If $T\in\mathcal{B}(\mathcal{H})$, the operator norm and numerical radius of $T$ are defined, respectively, as
$$\|T\|=\sup_{\|x\|=1}||Tx||=\sup_{\|x\|=\|y\|=1}|\left<Tx,y\right>|\;{\text{and}}\;\omega(T)=\sup_{\|x\|=1}|\left<Tx,x\right>|.$$
It is well known that $\omega(\cdot)$ is a norm on $\mathcal{B}(\mathcal{H})$, and this norm is equivalent to the operator norm, via \cite[Theorem 1.3-1]{gust}
\begin{equation}\label{14}
\frac{1}{2}\|T\|\leq\omega(T)\leq\|T\|.
\end{equation}
The first inequality in \eqref{14} becomes an equality when $T^2=O$, while the second inequality becomes an equality when $T$ is normal. Here $O$ denotes the zero operator.
Finding better estimates and more explicit relations among $\omega(T)$ and $\|T\|$ has been in the core interest of numerous researchers in this field. This is due to the importance of both quantities in understanding the geometry of Hilbert spaces and in other applications in operator theory, mathematical analysis, and physics. 

As one can see in \cite{4,gum,3}, operator matrices have played a significant role in obtaining further and sharper relations among the operator norm and the numerical radius. Here, we recall that if $A_i\in\mathcal{B}(\mathcal{H})$ for $i=1,2,3,4$, then the operator matrix $\left[\begin{array}{cc}A_1&A_2\\A_3&A_4\end{array}\right]\in\mathcal{B}(\mathcal{H}\oplus\mathcal{H})$. 

In this paper, we aim to present new estimates for the numerical radius and the operator norm of operator matrices. Such estimates enrich our knowledge and enlarge our helpful tools to tickle operator inequalities.

To proceed to our results, we will need some lemmas.

\begin{lemma}\label{lem_shebr}
\cite[Theorem 3.7]{3}
Let $X,Y,Z,W\in\mathcal{B}(\mathcal{H})$. Then
$$\omega\left(\left[\begin{array}{cc}X&Y\\Z&W\end{array}\right]\right)\geq \max\left(\omega(X),\omega(W),\frac{\omega(Y+Z)}{2},\frac{\omega(Y-Z)}{2}\right)$$ and
$$\omega\left(\left[\begin{array}{cc}X&Y\\Z&W\end{array}\right]\right)\leq \max\left(\omega(X),\omega(W)\right)+\frac{\omega(Y+Z)+\omega(Y-Z)}{2}.$$
\end{lemma}
Another useful lower bound for the numerical radius of an operator matrix can be stated as a pinching inequality as follows. It should be noted that this estimate is sharper than the first estimate in the above lemma as one can see in the proof of \cite[Theorem 3.7]{3}.
\begin{lemma}\label{lem_w_diag}
\cite[Lemma 3.1]{3}
Let $A_i\in\mathcal{B}(\mathcal{H})$ for $i=1,2,3,4$. Then
$$\omega\left(\left[\begin{array}{cc} A_1&A_2\\A_3&A_4\end{array}\right]\right)\geq \max\left\{\omega\left(\left[\begin{array}{cc} A_1&O\\O&A_4\end{array}\right]\right),\omega\left(\left[\begin{array}{cc} O&A_2\\A_3&O\end{array}\right]\right)\right\}.$$
\end{lemma}

Calculating the exact value of the numerical radius is, in general, not an easy task. However, for $2\times 2$ matrices, it can be easily found using the following formula.
\begin{lemma}\label{01}
\cite[Theorem 1]{johnson}
Let $X=\left[ \begin{matrix}
   a & b  \\
   c & d  \\
\end{matrix} \right]$ with $a,b,c,d\in \mathbb{R}$. Then
\[\omega \left( X \right)=\frac{\left| a+d \right|+\sqrt{{{\left( a-d \right)}^{2}}+{{\left( b+c \right)}^{2}}}}{2}.\]
\end{lemma}

We recall that if $T\in\mathcal{B}(\mathcal{H})$ is such that $\left<Tx,x\right>\geq 0$ for all $x\in\mathcal{H}$, then $T$ is said to be positive, and we write $T\geq O$ in this case. The following is a characterization of positivity  of a certain operator matrix, in terms of Cauchy-Shwarz type inequality.
\begin{lemma}\label{011}
\cite[Lemma 1]{04} Let $A,B,C\in \mathcal B\left( \mathcal H \right)$ be such that $A,B\geq O$.  Then 
\[\left[ \begin{matrix}
   A & C  \\
   {{C}^{*}} & B  \\
\end{matrix} \right]\ge O\Leftrightarrow {{\left| \left\langle Cx,y \right\rangle  \right|}^{2}}\le \left\langle Ax,x \right\rangle \left\langle By,y \right\rangle,\forall x,y\in\mathcal{H}.\]
\end{lemma}
We notice that for an operator matrix to be positive, the off-diagonal operators must be the conjugates of each other. In the following, we present a simple observation about the positivity of a certain transformation of operator matrices.
\begin{lemma}\label{5}
Let ${{T}_{i}}\in \mathcal B\left( \mathcal H \right)$ for $i=1,2,3,4$ be such that $\left[ \begin{matrix}
   {{T}_{1}} & {{T}_{2}}  \\
   T_{2}^{*} & {{T}_{3}}  \\
\end{matrix} \right]\geq O$. Then $\left[ \begin{matrix}
   t{{T}_{1}} & {{T}_{2}}  \\
   T_{2}^{*} & \frac{{{T}_{3}}}{t}  \\
\end{matrix} \right]\geq O$  for all $t>0$.
\end{lemma}
\begin{proof}
Recall that if $T\geq O$ and if $A\in\mathcal{B}(\mathcal{H})$ is arbitrary, then $ATA^*\geq O$. Then the result follows immediately noting that 
$$\left[ \begin{matrix}
   t{{T}_{1}} & {{T}_{2}}  \\
   T_{2}^{*} & \frac{{{T}_{3}}}{t}  \\
\end{matrix} \right]=\left[ \begin{matrix}
   \sqrt{t} & O  \\
   O & \sqrt{\frac{1}{t}}  \\
\end{matrix} \right]\left[ \begin{matrix}
   {{T}_{1}} & {{T}_{2}}  \\
   T_{2}^{*} & {{T}_{3}}  \\
\end{matrix} \right]\left[ \begin{matrix}
   \sqrt{t} & O  \\
   O & \sqrt{\frac{1}{t}}  \\
\end{matrix} \right].$$
\end{proof}

If $T\in\mathcal{B}(\mathcal{H})$ is arbitrary, then the real and imaginary parts of $T$ are defined respectively by $\mathcal{R}T=\frac{T+T^*}{2}$ and $\mathcal{I}T=\frac{T-T^*}{2\textup i}$. The Cartesian decomposition of $T$ is then defined by the form $T=\mathcal{R}T+{\textup i}\mathcal{I}T.$\\
Among the basic relations among the operator norm and the numerical radius, we have
\begin{equation}\label{eq_real}
\|\mathcal{R}T\|\leq \omega(T)\;{\text{and}}\;\|\mathcal{I}T\|\leq\omega(T).
\end{equation}
This follows immediately noting 
$$\left\| \mathcal RT \right\|=\left\| \frac{T+{{T}^{*}}}{2} \right\|=\omega \left( \frac{T+{{T}^{*}}}{2} \right)\le \frac{1}{2}\left( \omega \left( T \right)+\omega \left( {{T}^{*}} \right) \right)=\omega \left( T \right).$$ 
A more general class than that of positive operators is the class of accretive operators. Recall that an operator $T\in\mathcal{B}(\mathcal{H})$ is said to be accretive if $\mathcal{R}T\geq O$. Further, if $\mathcal{I}T\geq O$, then $T$ is said to be accretive-dissipative. The literature has verified that accretive and accretive-dissipative operators satisfy better bounds than arbitrary operators. We refer the reader to \cite{moradi, bks} as a sample of references emphasizing this observation.

Now, if $T:=\left[\begin{array}{cc} T_{11}&T_{12}\\ T_{21}&T_{22}\end{array}\right]\in\mathcal{B}(\mathcal{H}\oplus\mathcal{H})$, then the Cartesian decomposition of $T$ is \cite{gum}
 \begin{equation}\label{05}
\left[ \begin{matrix}
   {{T}_{11}} & {{T}_{12}}  \\
   {{T}_{21}} & {{T}_{22}}  \\
\end{matrix} \right]=\left[ \begin{matrix}
   {{A}_{11}} & {{A}_{12}}  \\
   {{A}_{21}} & {{A}_{22}}  \\
\end{matrix} \right]+\textup i\left[ \begin{matrix}
   {{B}_{11}} & {{B}_{12}}  \\
   {{B}_{21}} & {{B}_{22}}  \\
\end{matrix} \right],\text{ }{{A}_{12}}=A_{21}^{*},{{B}_{12}}=B_{21}^{*}
\end{equation}
 in which ${{T}_{jk}},{{A}_{jk}},{{B}_{jk}}\in \mathcal{B}(\mathcal{H})$ for $j,k=1,2$. We notice the necessity of the conditions ${{A}_{12}}=A_{21}^{*},{{B}_{12}}=B_{21}^{*}$ since the real and imaginary parts of an operator are necessarily self-adjoint. Also, we notice that $A_{ii}$ and $B_{ii}$ are self adjoint, for $i=1,2.$

In the next section, we present several relations for operator matrices partitioned as in \eqref{05}. This includes connections between the numerical radii and the operator norms of such operators and their components, including arbitrary and accretive-dissipative ones. After that, a more general discussion is presented about the operator norm of these operators. 
 
\section{Bounds for the numerical radius}

The following simple lemma will be needed to prove the upcoming result.
\begin{lemma}\label{04}
Let $T\in \mathcal{B}(\mathcal{H}\oplus\mathcal{H})$ be partitioned as in \eqref{05}. Then
\[\left\|{{A}_{12}} \right\|\le \omega \left( T \right).\]
\end{lemma}
\begin{proof}
Noting that $\left[\begin{array}{cc}A_{11}&A_{12}\\A_{12}^*&A_{22}\end{array}\right]$ is self-adjoint (and hence is normal), we have $\omega\left( \left[\begin{array}{cc}A_{11}&A_{12}\\A_{12}^*&A_{22}\end{array}\right] \right)=\left\| \left[\begin{array}{cc}A_{11}&A_{12}\\A_{12}^*&A_{22}\end{array}\right] \right\|$. Then
\[\begin{aligned}
   \left\| {{A}_{12}} \right\| &=\omega \left( \left[ \begin{matrix}
   O & {{A}_{12}}  \\
   A_{12}^{*} & O  \\
\end{matrix} \right] \right) \\ 
 & \le \omega \left( \left[ \begin{matrix}
   {{A}_{11}} & {{A}_{12}}  \\
   A_{12}^{*} & {{A}_{22}}  \\
\end{matrix} \right] \right) \quad \text{(by\;Lemma\;\ref{lem_w_diag})} \\
&=\left\|\left[ \begin{matrix}
   {{A}_{11}} & {{A}_{12}}  \\
   A_{12}^{*} & {{A}_{22}}  \\
\end{matrix} \right]\right\|\\
&\leq \omega(T)\quad({\text{by}}\;\eqref{eq_real}).
\end{aligned}\]
 This completes the proof.
\end{proof}
In the proof of the above result, we have used the fact that $\|\mathcal{R}T\|\leq \omega(T)$

\begin{theorem}\label{06}
Let $T\in \mathcal{B}(\mathcal{H}\oplus\mathcal{H})$ be partitioned as in \eqref{05}. Then
\[\frac{1}{4}\left\| {{\left| {{T}_{12}} \right|}^{2}}+{{\left| T_{21}^{*} \right|}^{2}} \right\|\le {{\omega }^{2}}\left( T \right).\]
\end{theorem}
\begin{proof}
It follows from Lemma \ref{04} that
\begin{equation}\label{11}
\left\| {{A}_{12}} \right\|\le \omega \left( \mathcal RT \right)\;\text{ and }\;\left\| {{B}_{12}} \right\|\le \omega \left( \mathcal IT \right).
\end{equation}

We then obtain
\[\begin{aligned}
   \left\| {{\left| {{T}_{12}} \right|}^{2}}+{{\left| T_{21}^{*} \right|}^{2}} \right\|&=\left\| {{\left| {{A}_{12}}+\textup i{{B}_{12}} \right|}^{2}}+{{\left| {{A}_{12}}-\textup i{{B}_{12}} \right|}^{2}} \right\| \\ 
 & =2\left\| {{\left| {{A}_{12}} \right|}^{2}}+{{\left| {{B}_{12}} \right|}^{2}} \right\| \\ 
 & \le 2\left( {{\left\| {{A}_{12}} \right\|}^{2}}+{{\left\| {{B}_{12}} \right\|}^{2}} \right) \quad \text{(by the triangle inequality)}\\ 
 & \le 2\left( {{\omega }^{2}}\left( \mathcal RT \right)+{{\omega }^{2}}\left( \mathcal IT \right) \right) \quad \text{(by \eqref{11})}\\ 
 & \le 4{{\omega }^{2}}\left( T \right) \quad \text{(by \eqref{eq_real})},
\end{aligned}\]
and the proof is complete.
\end{proof}

In the rest of this section, we treat accretive-dissipative operator matrices. In \cite{03}, it has been shown that if $A=\left[ \begin{matrix}
   {{A}_{11}} & {{A}_{12}}  \\
   A_{12}^{*} & {{A}_{22}}  \\
\end{matrix} \right]\geq O$, then
\begin{equation}\label{8}
2\left\| {{A}_{12}} \right\|\le \left\| \left[ \begin{matrix}
   {{A}_{11}} & {{A}_{12}}  \\
   A_{12}^{*} & {{A}_{22}}  \\
\end{matrix} \right] \right\|.
\end{equation}
The following result presents a possible numerical radius version for accretive-dissipative operators. As we mentioned in the introduction, it is natural that accretive-dissipative operators satisfy better bounds than arbitrary ones. We leave the proof of Theorem \ref{08} to the reader. Of course, this result must be compared with Theorem \ref{06}.
\begin{theorem}\label{08}
Let $T\in  \mathcal{B}(\mathcal{H}\oplus\mathcal{H})$ be accretive-dissipative partitioned as in \eqref{05}. Then
\[\left\| {{\left| {{T}_{12}} \right|}^{2}}+{{\left| T_{21}^{*} \right|}^{2}} \right\|\le {{\omega }^{2}}\left( T \right).\]
\end{theorem}

It should be noted that the inequalities in Theorems \ref{06}, and \ref{08} are sharp. To show the sharpness of these theorems, consider the $2\times 2$ matrices  $\left[ \begin{matrix}
   0 & 1  \\
   0 & 0  \\
\end{matrix} \right]$, and $\left[ \begin{matrix}
   1+\textup i & 1-\textup i  \\
   1-\textup i & 1+\textup i  \\
\end{matrix} \right]$, respectively (for each theorem).

For a positive  block matrix $\left[ \begin{matrix}
   {{A}_{11}} & {{A}_{12}}  \\
   A_{12}^{*} & {{A}_{22}}  \\
\end{matrix} \right]$, it is well known that \cite{bl}
\begin{equation}\label{09}
\left\| \left[ \begin{matrix}
   {{A}_{11}} & {{A}_{12}}  \\
   A_{12}^{*} & {{A}_{22}}  \\
\end{matrix} \right] \right\|\le \left\| {{A}_{11}} \right\|+\left\| {{A}_{22}} \right\|.
\end{equation}
Employing this inequality, we can obtain the following result.

We conclude this section with the following  upper bounds for $\omega \left( {{T}_{12}} \right)$.
\begin{theorem}
Let $T\in  \mathcal{B}(\mathcal{H}\oplus\mathcal{H})$ be accretive-dissipative partitioned as in \eqref{05}. Then
\begin{equation}\label{6}
\omega \left( {{T}_{12}} \right)\le \frac{1}{2}\left\| {{A}_{11}}+{{A}_{22}}+{{B}_{11}}+{{B}_{22}} \right\|,
\end{equation}
and
\begin{equation}\label{7}
\omega \left( {{T}_{12}} \right)\le \sqrt{\left\| {{A}_{11}}+{{B}_{11}} \right\|\left\| {{A}_{22}}+{{B}_{22}} \right\|}.
\end{equation}
Further, both inequalities are sharp.
\end{theorem}
\begin{proof}
Since $T$ is an accretive-dissipative matrix, then, by Lemma \ref{011}, we have 
\[\left| \left\langle {{A}_{12}}x,x \right\rangle  \right|\le \sqrt{\left\langle {{A}_{11}}x,x \right\rangle \left\langle {{A}_{22}}x,x \right\rangle }\;\text{ and }\;\left| \left\langle {{B}_{12}}x,x \right\rangle  \right|\le \sqrt{\left\langle {{B}_{11}}x,x \right\rangle \left\langle {{B}_{22}}x,x \right\rangle }\]
for any unit vector $x\in \mathcal H$. Therefore,
\begin{equation}\label{010}
\begin{aligned}
   \left| \left\langle {{T}_{12}}x,x \right\rangle  \right|&=\left| \left\langle \left( {{A}_{12}}+\textup i{{B}_{12}} \right)x,x \right\rangle  \right|\\
   &=\left| \left\langle {{A}_{12}}x,x \right\rangle +\textup i\left\langle {{B}_{12}}x,x \right\rangle  \right| \\ 
 & \le \left| \left\langle {{A}_{12}}x,x \right\rangle  \right|+\left| \left\langle {{B}_{12}}x,x \right\rangle  \right| \\ 
 & \le \sqrt{\left\langle {{A}_{11}}x,x \right\rangle \left\langle {{A}_{22}}x,x \right\rangle }+\sqrt{\left\langle {{B}_{11}}x,x \right\rangle \left\langle {{B}_{22}}x,x \right\rangle }. 
\end{aligned}
\end{equation}
On the other hand, by the arithmetic-geometric mean inequality, we know that
\[\sqrt{\left\langle {{A}_{11}}x,x \right\rangle \left\langle {A}_{22}x,x \right\rangle }+\sqrt{\left\langle {{B}_{11}}x,x \right\rangle \left\langle {{B}_{22}}x,x \right\rangle }\le \frac{1}{2}\left\langle \left( {{A}_{11}}+{{A}_{22}}+{{B}_{11}}+{{B}_{22}} \right)x,x \right\rangle.\]
Combining this with the inequality \eqref{010}, gives the inequality \eqref{6}.

For the inequality \eqref{7}, sine $T$ is accretive-dissipative, it follows from Lemma \ref{5} that $\left[ \begin{matrix}
   t{{T}_{11}} & {{T}_{12}}  \\
   {{T}_{21}} & \frac{{{T}_{22}}}{t}  \\
\end{matrix} \right]$ is accretive-dissipative  for all $t>0$. Now applying the inequality \eqref{6} to $S=\left[ \begin{matrix}
   t{{A}_{11}} & {{A}_{12}}  \\
   A_{12}^{*} & \frac{{{A}_{22}}}{t}  \\
\end{matrix} \right]+\textup i\left[ \begin{matrix}
   t{{B}_{11}} & {{B}_{12}}  \\
   B_{12}^{*} & \frac{{{B}_{22}}}{t}  \\
\end{matrix} \right]$, we have
\[\begin{aligned}
   \omega \left( {{T}_{12}} \right)&\le \frac{1}{2}\left\| t{{A}_{11}}+\frac{{{A}_{22}}}{t}+t{{B}_{11}}+\frac{{{B}_{22}}}{t} \right\| \\ 
 & \le \frac{1}{2}\left( t\left\| {{A}_{11}}+{{B}_{11}} \right\|+\frac{1}{t}\left\| {{A}_{22}}+{{B}_{22}} \right\| \right)  
\end{aligned}\]
 for all $t>0$. Now taking the minimum over all $t>0$, we have
\[\omega \left( {{T}_{12}} \right)\le \sqrt{\left\| {{A}_{11}}+{{B}_{11}} \right\|\left\| {{A}_{22}}+{{B}_{22}} \right\|},\]
as required. Sharpness of the inequalities can be seen using the matrix $T=\left[\begin{array}{cc}1&1\\1&1\end{array}\right].$
\end{proof}

On the other hand, a lower bound for $\omega(T)$ in terms of its real and imaginary parts can be stated as follows.

\begin{theorem}
Let $T\in \mathcal{B}(\mathcal{H}\oplus\mathcal{H})$ be  partitioned as in \eqref{05}. Then
\[\max \left( \alpha ,\beta  \right)\le \omega\left(T\right),\]
where
\[\alpha =\max \left( \left\| {{A}_{11}} \right\|,\left\| {{A}_{22}} \right\|,\frac{\left\| {{A}_{12}}+A_{12}^{*} \right\|}{2},\frac{\left\| {{A}_{12}}-A_{12}^{*} \right\|}{2} \right),\]
and
\[\beta =\max \left( \left\| {{B}_{11}} \right\|,\left\| {{B}_{22}} \right\|,\frac{\left\| {{B}_{12}}+B_{12}^{*} \right\|}{2},\frac{\left\| {{B}_{12}}-B_{12}^{*} \right\|}{2} \right).\]
Further, the inequality is sharp.
\end{theorem}
\begin{proof}
We know that $\left\| \mathcal RT \right\|,\left\| \mathcal IT \right\|\le \omega(T)$. In consequence,
\begin{equation}\label{3}
\left\| \left[ \begin{matrix}
   {{A}_{11}} & {{A}_{12}}  \\
   A_{12}^{*} & {{A}_{22}}  \\
\end{matrix} \right] \right\|,\left\| \left[ \begin{matrix}
   {{B}_{11}} & {{B}_{12}}  \\
   B_{12}^{*} & {{B}_{22}}  \\
\end{matrix} \right] \right\|\le \omega(T).
\end{equation}
On the other hand, by Lemma \ref{lem_shebr}, we have
\[\begin{aligned}
  & \max \left( \left\| {{A}_{11}} \right\|,\left\| {{A}_{22}} \right\|,\frac{\left\| {{A}_{12}}+A_{12}^{*} \right\|}{2},\frac{\left\| {{A}_{12}}-A_{12}^{*} \right\|}{2} \right) \\ 
 & =\max \left( \omega \left( {{A}_{11}} \right),\omega \left( {{A}_{22}} \right),\frac{\omega \left( {{A}_{12}}+A_{12}^{*} \right)}{2},\frac{\omega \left( {{A}_{12}}-A_{12}^{*} \right)}{2} \right) \\ 
 & \le \omega \left( \left[ \begin{matrix}
   {{A}_{11}} & {{A}_{12}}  \\
   A_{12}^{*} & {{A}_{22}}  \\
\end{matrix} \right] \right) \\ 
 & =\left\| \left[ \begin{matrix}
   {{A}_{11}} & {{A}_{12}}  \\
   A_{12}^{*} & {{A}_{22}}  \\
\end{matrix} \right] \right\|,  
\end{aligned}\]
i.e.,
\begin{equation}\label{4}
\max \left( \left\| {{A}_{11}} \right\|,\left\| {{A}_{22}} \right\|,\frac{\left\| {{A}_{12}}+A_{12}^{*} \right\|}{2},\frac{\left\| {{A}_{12}}-A_{12}^{*} \right\|}{2} \right)\le \left\| \left[ \begin{matrix}
   {{A}_{11}} & {{A}_{12}}  \\
   A_{12}^{*} & {{A}_{22}}  \\
\end{matrix} \right] \right\|.
\end{equation}
Combining the inequalities \eqref{3} and \eqref{4}, gives
\begin{equation}\label{9}
\max \left( \left\| {{A}_{11}} \right\|,\left\| {{A}_{22}} \right\|,\frac{\left\| {{A}_{12}}+A_{12}^{*} \right\|}{2},\frac{\left\| {{A}_{12}}-A_{12}^{*} \right\|}{2} \right)\le \omega(T).
\end{equation}
In a similar way, 
\begin{equation}\label{10}
\max \left( \left\| {{B}_{11}} \right\|,\left\| {{B}_{22}} \right\|,\frac{\left\| {{B}_{12}}+B_{12}^{*} \right\|}{2},\frac{\left\| {{B}_{12}}-B_{12}^{*} \right\|}{2} \right)\le \omega(T).
\end{equation}
The desired inequality follows by \eqref{9} and \eqref{10}. Sharpness of the obtained inequality can be verified using the matrix $T=\left[\begin{array}{cc}1&0\\0&0\end{array}\right]$. This completes the proof.
\end{proof}

\begin{corollary}
Let $T\in \mathcal{B}(\mathcal{H}\oplus\mathcal{H})$ be accretive-dissipative partitioned as in \eqref{05}. Then
\[2\max \left( \left\| {{A}_{12}} \right\|,\left\| {{B}_{12}} \right\| \right)\le \omega(T).\]
\end{corollary}
\begin{proof}
By the inequalities \eqref{3} and \eqref{8}, we infer that
\[2\left\| {{A}_{12}} \right\|\le \left\| \left[ \begin{matrix}
   {{A}_{11}} & {{A}_{12}}  \\
   A_{12}^{*} & {{A}_{22}}  \\
\end{matrix} \right] \right\|\le \omega\left( \left[ \begin{matrix}
   {{T}_{11}} & {{T}_{12}}  \\
   {{T}_{21}} & {{T}_{22}}  \\
\end{matrix} \right] \right)\]
i.e.,
\[2\left\| {{A}_{12}} \right\|\le \omega(T).\]
In a similar way we can prove
\[2\left\| {{B}_{12}} \right\|\le \omega(T).\]
These inequalities prove the assertion.
\end{proof}

\section{Bounds for the operator norm}
In the previous section, we presented upper bounds for the numerical radius of operator matrices and their components. This section proves some upper bounds for general $2\times 2$ operator matrices.

We should remark here that if $T:=\left[\begin{array}{cc}A&B\\B^*&C\end{array}\right]$ is positive, then $\|T\|\leq \|A\|+\|C\|$. On the other hand, if $T\geq O$ and $B$ is self-adjoint, then $\|T\|\leq \|A+C\|$; see \cite{hiro}.\\
In what follows, we present new bounds for general $2\times 2$ operator matrices.

In this result, the notation $r(\cdot)$ refers to the spectral radius. We recall here that when $T$ is self-adjoint then $r(T)=\omega(T)=\|T\|$. We also recall that \cite[Lemma 2.1]{5}
\begin{equation}\label{eq_r}
r\left(\left[\begin{array}{cc} A&B\\ C&D\end{array}\right]\right)\leq r\left(\left[\begin{array}{cc}\|A\|&\|B\| \\ \|C\| &\|D\| \end{array}\right]\right).
\end{equation}

\begin{theorem} 
Let $T\in\mathcal{B}(\mathcal{H}\oplus\mathcal{H})$ be partitioned as in \eqref{05}. Then
\[\begin{aligned}
   \left\| \left[ \begin{matrix}
   {{T}_{11}} & {{T}_{12}}  \\
   {{T}_{21}} & {{T}_{22}}  \\
\end{matrix} \right] \right\|&\le \frac{1}{2}\left( \left\| {{A}_{11}} \right\|+\left\| {{A}_{22}} \right\|+\sqrt{{{\left( \left\| {{A}_{11}} \right\|-\left\| {{A}_{22}} \right\| \right)}^{2}}+4{{\left\| {{A}_{12}} \right\|}^{2}}} \right) \\ 
 &\quad +\frac{1}{2}\left( \left\| {{B}_{11}} \right\|+\left\| {{B}_{22}} \right\|+\sqrt{{{\left( \left\| {{B}_{11}} \right\|-\left\| {{B}_{22}} \right\| \right)}^{2}}+4{{\left\| {{B}_{12}} \right\|}^{2}}} \right).  
\end{aligned}\]
\end{theorem}
\begin{proof}
Using the triangle inequality, then noting that $\left[ \begin{matrix}
   {{A}_{11}} & {{A}_{12}}  \\
   A_{12}^{*} & {{A}_{22}}  \\
\end{matrix} \right]$ and $\left[ \begin{matrix}
   {{B}_{11}} & {{B}_{12}}  \\
   B_{12}^{*} & {{B}_{22}}  \\
\end{matrix} \right]$ are self-adjoint, we have 
\[\begin{aligned}
  & \left\| \left[ \begin{matrix}
   {{T}_{11}} & {{T}_{12}}  \\
   {{T}_{21}} & {{T}_{22}}  \\
\end{matrix} \right] \right\| \\ 
 & =\left\| \left[ \begin{matrix}
   {{A}_{11}} & {{A}_{12}}  \\
   A_{12}^{*} & {{A}_{22}}  \\
\end{matrix} \right]+\textup i\left[ \begin{matrix}
   {{B}_{11}} & {{B}_{12}}  \\
   B_{12}^{*} & {{B}_{22}}  \\
\end{matrix} \right] \right\| \\ 
 & \le \left\| \left[ \begin{matrix}
   {{A}_{11}} & {{A}_{12}}  \\
   A_{12}^{*} & {{A}_{22}}  \\
\end{matrix} \right] \right\|+\left\| \left[ \begin{matrix}
   {{B}_{11}} & {{B}_{12}}  \\
   B_{12}^{*} & {{B}_{22}}  \\
\end{matrix} \right] \right\| \\ 
 & =r\left( \left[ \begin{matrix}
   {{A}_{11}} & {{A}_{12}}  \\
   A_{12}^{*} & {{A}_{22}}  \\
\end{matrix} \right] \right)+r\left( \left[ \begin{matrix}
   {{B}_{11}} & {{B}_{12}}  \\
   B_{12}^{*} & {{B}_{22}}  \\
\end{matrix} \right] \right) \\ 
 & \le r\left( \left[ \begin{matrix}
   \left\| {{A}_{11}} \right\| & \left\| {{A}_{12}} \right\|  \\
   \left\| A_{12}^{*} \right\| & \left\| {{A}_{22}} \right\|  \\
\end{matrix} \right] \right)+r\left( \left[ \begin{matrix}
   \left\| {{B}_{11}} \right\| & \left\| {{B}_{12}} \right\|  \\
   \left\| B_{12}^{*} \right\| & \left\| {{B}_{22}} \right\|  \\
\end{matrix} \right] \right) \quad({\text{by}}\;\eqref{eq_r})\\ 
 & =\frac{1}{2}\left( \left\| {{A}_{11}} \right\|+\left\| {{A}_{22}} \right\|+\sqrt{{{\left( \left\| {{A}_{11}} \right\|-\left\| {{A}_{22}} \right\| \right)}^{2}}+4{{\left\| {{A}_{12}} \right\|}^{2}}} \right) \\ 
 &\quad +\frac{1}{2}\left( \left\| {{B}_{11}} \right\|+\left\| {{B}_{22}} \right\|+\sqrt{{{\left( \left\| {{B}_{11}} \right\|-\left\| {{B}_{22}} \right\| \right)}^{2}}+4{{\left\| {{B}_{12}} \right\|}^{2}}} \right), 
\end{aligned}\]
where we have used Lemma \ref{01} to obtain the last equation. This completes the proof.
\end{proof}

In the following results, we list some weighted upper bounds for $ \left\| \left[ \begin{matrix}
   {{T}_{1}} & {{T}_{2}}  \\
   {{T}_{3}} & {{T}_{4}}  \\
\end{matrix} \right] \right\|$.
\begin{theorem}\label{1}
Let ${{T}_{i}}\in \mathcal B\left( \mathcal H \right)$ for $i=1,2,3,4$. Then for any $0\le t\le 1$,
\[\begin{aligned}
   \left\| \left[ \begin{matrix}
   {{T}_{1}} & {{T}_{2}}  \\
   {{T}_{3}} & {{T}_{4}}  \\
\end{matrix} \right] \right\|&\le \frac{1}{2}\max \left( \left\| {{\left| {{T}_{1}} \right|}^{2t}}+{{\left| {{T}_{3}} \right|}^{2t}} \right\|,\left\| {{\left| {{T}_{4}} \right|}^{2t}}+{{\left| {{T}_{2}} \right|}^{2t}} \right\| \right) \\ 
 &\quad +\frac{1}{2}\max \left( \left\| {{\left| T_{1}^{*} \right|}^{2\left( 1-t \right)}}+{{\left| T_{3}^{*} \right|}^{2\left( 1-t \right)}} \right\|,\left\| {{\left| T_{4}^{*} \right|}^{2\left( 1-t \right)}}+{{\left| T_{2}^{*} \right|}^{2\left( 1-t \right)}} \right\| \right). 
\end{aligned}\]
\end{theorem}
\begin{proof}
Let $\mathbf x=\left[ \begin{matrix}
   {{x}_{1}}  \\
   {{x}_{2}}  \\
\end{matrix} \right]$ and $\mathbf y=\left[ \begin{matrix}
   {{y}_{1}}  \\
   {{y}_{2}}  \\
\end{matrix} \right]$ be any two unit vectors in $\mathcal H\oplus \mathcal H$. Then
\begin{align}
  & \left| \left\langle \left[ \begin{matrix}
   {{T}_{1}} & {{T}_{2}}  \\
   {{T}_{3}} & {{T}_{4}}  \\
\end{matrix} \right]\mathbf x,\mathbf y \right\rangle  \right| \nonumber\\ 
 & =\left| \left\langle \left[ \begin{matrix}
   {{T}_{1}} & O  \\
   O & {{T}_{4}}  \\
\end{matrix} \right]\mathbf x,\mathbf y \right\rangle +\left\langle \left[ \begin{matrix}
   O & {{T}_{2}}  \\
   {{T}_{3}} & O  \\
\end{matrix} \right]\mathbf x,\mathbf y \right\rangle  \right| \nonumber\\ 
 & \le \left| \left\langle \left[ \begin{matrix}
   {{T}_{1}} & O  \\
   O & {{T}_{4}}  \\
\end{matrix} \right]\mathbf x,\mathbf y \right\rangle  \right|+\left| \left\langle \left[ \begin{matrix}
   O & {{T}_{2}}  \\
   {{T}_{3}} & O  \\
\end{matrix} \right]\mathbf x,\mathbf y \right\rangle  \right|\quad \text{(by the triangle inequality)} \nonumber\\ 
 & \le \sqrt{\left\langle {{\left| \left[ \begin{matrix}
   {{T}_{1}} & O  \\
   O & {{T}_{4}}  \\
\end{matrix} \right] \right|}^{2t}}\mathbf x,\mathbf x \right\rangle \left\langle {{\left| {{\left[ \begin{matrix}
   {{T}_{1}} & O  \\
   O & {{T}_{4}}  \\
\end{matrix} \right]}^{*}} \right|}^{2\left( 1-t \right)}}\mathbf y,\mathbf y \right\rangle }\nonumber \\ 
 &\quad +\sqrt{\left\langle {{\left| \left[ \begin{matrix}
   O & {{T}_{2}}  \\
   {{T}_{3}} & O  \\
\end{matrix} \right] \right|}^{2t}}\mathbf x,\mathbf x \right\rangle \left\langle {{\left| {{\left[ \begin{matrix}
   O & {{T}_{2}}  \\
   {{T}_{3}} & O  \\
\end{matrix} \right]}^{*}} \right|}^{2\left( 1-t \right)}}\mathbf y,\mathbf y \right\rangle } \quad\text{(by the mixed-Schwarz inequality \cite{kato})}\nonumber\\ 
 & =\sqrt{\left\langle \left[ \begin{matrix}
   {{\left| {{T}_{1}} \right|}^{2t}} & O  \\
   O & {{\left| {{T}_{4}} \right|}^{2t}}  \\
\end{matrix} \right]\mathbf x,\mathbf x \right\rangle \left\langle \left[ \begin{matrix}
   {{\left| T_{1}^{*} \right|}^{2\left( 1-t \right)}} & O  \\
   O & {{\left| T_{4}^{*} \right|}^{2\left( 1-t \right)}}  \\
\end{matrix} \right]\mathbf y,\mathbf y \right\rangle }\nonumber \\ 
 &\quad +\sqrt{\left\langle \left[ \begin{matrix}
   {{\left| {{T}_{3}} \right|}^{2t}} & O  \\
   O & {{\left| {{T}_{2}} \right|}^{2t}}  \\
\end{matrix} \right]\mathbf x,\mathbf x \right\rangle \left\langle \left[ \begin{matrix}
   {{\left| T_{3}^{*} \right|}^{2\left( 1-t \right)}} & O  \\
   O & {{\left| T_{2}^{*} \right|}^{2\left( 1-t \right)}}  \\
\end{matrix} \right]\mathbf y,\mathbf y \right\rangle } \label{xxx}\\ 
 & \le \frac{1}{2}\left\langle \left[ \begin{matrix}
   {{\left| {{T}_{1}} \right|}^{2t}}+{{\left| {{T}_{3}} \right|}^{2t}} & O  \\
   O & {{\left| {{T}_{4}} \right|}^{2t}}+{{\left| {{T}_{2}} \right|}^{2t}}  \\
\end{matrix} \right]\mathbf x,\mathbf x \right\rangle \nonumber \\ 
 & \quad+\frac{1}{2}\left\langle \left[ \begin{matrix}
   {{\left| T_{1}^{*} \right|}^{2\left( 1-t \right)}}+{{\left| T_{3}^{*} \right|}^{2\left( 1-t \right)}} & O  \\
   O & {{\left| T_{4}^{*} \right|}^{2\left( 1-t \right)}}+{{\left| T_{2}^{*} \right|}^{2\left( 1-t \right)}}  \\
\end{matrix} \right]\mathbf y,\mathbf y \right\rangle\nonumber\\
&\qquad\text{(by the arithmetic-geometric mean inequality)} \nonumber \\
 & \le \frac{1}{2}\left\| \left[ \begin{matrix}
   {{\left| {{T}_{1}} \right|}^{2t}}+{{\left| {{T}_{3}} \right|}^{2t}} & O  \\
   O & {{\left| {{T}_{4}} \right|}^{2t}}+{{\left| {{T}_{2}} \right|}^{2t}}  \nonumber\\
\end{matrix} \right] \right\|\nonumber \\ 
 &\quad +\frac{1}{2}\left\| \left[ \begin{matrix}
   {{\left| T_{1}^{*} \right|}^{2\left( 1-t \right)}}+{{\left| T_{3}^{*} \right|}^{2\left( 1-t \right)}} & O  \\
   O & {{\left| T_{4}^{*} \right|}^{2\left( 1-t \right)}}+{{\left| T_{2}^{*} \right|}^{2\left( 1-t \right)}}  \\
\end{matrix} \right] \right\| \nonumber\\ 
 & =\frac{1}{2}\max \left( \left\| {{\left| {{T}_{1}} \right|}^{2t}}+{{\left| {{T}_{3}} \right|}^{2t}} \right\|,\left\| {{\left| {{T}_{4}} \right|}^{2t}}+{{\left| {{T}_{2}} \right|}^{2t}} \right\| \right) \nonumber\\ 
 &\quad +\frac{1}{2}\max \left( \left\| {{\left| T_{1}^{*} \right|}^{2\left( 1-t \right)}}+{{\left| T_{3}^{*} \right|}^{2\left( 1-t \right)}} \right\|,\left\| {{\left| T_{4}^{*} \right|}^{2\left( 1-t \right)}}+{{\left| T_{2}^{*} \right|}^{2\left( 1-t \right)}} \right\| \right),  \nonumber
\end{align}
which yields that
\[\begin{aligned}
   \left\| \left[ \begin{matrix}
   {{T}_{1}} & {{T}_{2}}  \\
   {{T}_{3}} & {{T}_{4}}  \\
\end{matrix} \right] \right\|&\le \frac{1}{2}\max \left( \left\| {{\left| {{T}_{1}} \right|}^{2t}}+{{\left| {{T}_{3}} \right|}^{2t}} \right\|,\left\| {{\left| {{T}_{4}} \right|}^{2t}}+{{\left| {{T}_{2}} \right|}^{2t}} \right\| \right) \\ 
 &\quad +\frac{1}{2}\max \left( \left\| {{\left| T_{1}^{*} \right|}^{2\left( 1-t \right)}}+{{\left| T_{3}^{*} \right|}^{2\left( 1-t \right)}} \right\|,\left\| {{\left| T_{4}^{*} \right|}^{2\left( 1-t \right)}}+{{\left| T_{2}^{*} \right|}^{2\left( 1-t \right)}} \right\| \right), 
\end{aligned}\]
as required.
\end{proof}

In fact, a more general form of Theorem \ref{1} can be stated in terms of functions $f,g:[0,\infty)$ satisfying $f(x)g(x)=x$. In this case, we get
\begin{align*}
\left\|\left[\begin{array}{cc} T_1&T_2\\T_3&T_4\end{array}\right]\right\|\leq &\frac{1}{2}\max\left(\|f(|T_1|^2)+f(|T_3|^2)\|,\|f(|T_4|^2)+f(|T_2|^2)\|\right)\\
&+\frac{1}{2}\max\left(\|g(|T_1^*|^2)+g(|T_3^*|^2)\|,\|g(|T_4^*|^2)+g(|T_2^*|^2)\|\right).
\end{align*}
In Theorem \ref{thm_2} below, we present our result in this form for the reader's convenience.
\begin{remark}
 From Theorem \ref{1} it is immediate that
\[\begin{aligned}
   \left\| \left[ \begin{matrix}
   {{T}_{1}} & {{T}_{2}}  \\
   T_{2}^{*} & {{T}_{3}}  \\
\end{matrix} \right] \right\|&\le \frac{1}{2}\max \left( \left\| {{\left| {{T}_{1}} \right|}^{2t}}+{{\left| {{T}_{2}}^* \right|}^{2t}} \right\|,\left\| {{\left| T_{2} \right|}^{2t}}+{{\left| {{T}_{3}} \right|}^{2t}} \right\| \right) \\ 
 &\quad +\frac{1}{2}\max \left( \left\| {{\left| T_{1}^{*} \right|}^{2\left( 1-t \right)}}+{{\left| T_{2} \right|}^{2\left( 1-t \right)}} \right\|,\left\| {{\left| {{T}_{2}}^* \right|}^{2\left( 1-t \right)}}+{{\left| T_{3}^{*} \right|}^{2\left( 1-t \right)}} \right\| \right). 
\end{aligned}\]
\end{remark}

\begin{remark}
It follows from Theorem \ref{1} that
	\[\begin{aligned}
   \max \left( \left\| {{T}_{1}}+{{T}_{2}} \right\|,\left\| {{T}_{1}}-{{T}_{2}} \right\| \right)&=\left\| \left[ \begin{matrix}
   {{T}_{1}} & {{T}_{2}}  \\
   {{T}_{2}} & {{T}_{1}}  \\
\end{matrix} \right] \right\| \\ 
 & \le \frac{1}{2}\left( \left\| {{\left| {{T}_{1}} \right|}^{2t}}+{{\left| {{T}_{2}} \right|}^{2t}} \right\|+\left\| {{\left| T_{1}^{*} \right|}^{2\left( 1-t \right)}}+{{\left| T_{2}^{*} \right|}^{2\left( 1-t \right)}} \right\| \right).  
\end{aligned}\]
In particular (see \cite[Theorem 6]{had})
	\[\max \left( \left\| {{T}_{1}}+{{T}_{2}} \right\|,\left\| {{T}_{1}}-{{T}_{2}} \right\| \right)\le \frac{1}{2}\left( \left\| \left| {{T}_{1}} \right|+\left| {{T}_{2}} \right| \right\|+\left\| \left| T_{1}^{*} \right|+\left| T_{2}^{*} \right| \right\| \right).\]
Notice that the inequality 
	\[\left\| {{T}_{1}}+{{T}_{2}} \right\|\le \left\| \left| {{T}_{1}} \right|+\left| {{T}_{2}} \right| \right\|\]
is not true for all ${{T}_{1}},{{T}_{2}}$; see  \cite[(1.42)]{1}. The above inequality holds when ${{T}_{1}},{{T}_{2}}$ are normal operator.

\end{remark}

Another upper bound for $\left\| \left[ \begin{matrix}
   {{T}_{1}} & {{T}_{2}}  \\
   T_{3} & {{T}_{4}}  \\
\end{matrix} \right] \right\|$ can be found as follows.
\begin{theorem}\label{thm_2}
Let ${{T}_{i}}\in \mathcal B\left( \mathcal H \right)$ for $i=1,2,3,4$. If $f, g$ are nonnegative continuous functions on $\left[ 0,\infty  \right)$ satisfying $f(x)g(x) =x$, $(x \ge0)$, then
\[\begin{aligned}
   \left\| \left[ \begin{matrix}
   {{T}_{1}} & {{T}_{2}}  \\
   {{T}_{3}} & {{T}_{4}}  \\
\end{matrix} \right] \right\|&\le \frac{1}{2}\left( \left\| {{f}^{2}}\left( \left| {{T}_{1}} \right| \right)+{{f}^{2}}\left( \left| {{T}_{3}} \right| \right) \right\|+\left\| {{g}^{2}}\left( \left| T_{1}^{*} \right| \right)+{{g}^{2}}\left( \left| T_{2}^{*} \right| \right) \right\| \right. \\ 
 &\qquad \left. +\left\| {{f}^{2}}\left( \left| {{T}_{2}} \right| \right)+{{f}^{2}}\left( \left| {{T}_{4}} \right| \right) \right\|+\left\| {{g}^{2}}\left( \left| T_{3}^{*} \right| \right)+{{g}^{2}}\left( \left| T_{4}^{*} \right| \right) \right\| \right). 
\end{aligned}\]
In particular, for any $0\le t\le 1$,
\[\begin{aligned}
   \left\| \left[ \begin{matrix}
   {{T}_{1}} & {{T}_{2}}  \\
   {{T}_{3}} & {{T}_{4}}  \\
\end{matrix} \right] \right\|&\le \frac{1}{2}\left( \left\| {{\left| {{T}_{1}} \right|}^{2t}}+{{\left| {{T}_{3}} \right|}^{2t}} \right\|+\left\| {{\left| T_{1}^{*} \right|}^{2\left( 1-t \right)}}+{{\left| T_{2}^{*} \right|}^{2\left( 1-t \right)}} \right\| \right. \\ 
 &\qquad \left. +\left\| {{\left| {{T}_{2}} \right|}^{2t}}+{{\left| {{T}_{4}} \right|}^{2t}} \right\|+\left\| {{\left| T_{3}^{*} \right|}^{2\left( 1-t \right)}}+{{\left| T_{4}^{*} \right|}^{2\left( 1-t \right)}} \right\| \right).
\end{aligned}\]
\end{theorem}
\begin{proof}
Let $\mathbf x=\left[ \begin{matrix}
   {{x}_{1}}  \\
   {{x}_{2}}  \\
\end{matrix} \right]$ and $\mathbf y=\left[ \begin{matrix}
   {{y}_{1}}  \\
   {{y}_{2}}  \\
\end{matrix} \right]$ be any two unit vectors in $\mathcal H\oplus \mathcal{H}$. Then
\[\begin{aligned}
  & \left| \left\langle \left[ \begin{matrix}
   {{T}_{1}} & {{T}_{2}}  \\
   {{T}_{3}} & {{T}_{4}}  \\
\end{matrix} \right]\mathbf x,\mathbf y \right\rangle  \right| \\ 
 & =\left| \left\langle \left[ \begin{matrix}
   {{T}_{1}} & {{T}_{2}}  \\
   {{T}_{3}} & {{T}_{4}}  \\
\end{matrix} \right]\left[ \begin{matrix}
   {{x}_{1}}  \\
   {{x}_{2}}  \\
\end{matrix} \right],\left[ \begin{matrix}
   {{y}_{1}}  \\
   {{y}_{2}}  \\
\end{matrix} \right] \right\rangle  \right| \\ 
 & =\left| \left\langle {{T}_{1}}{{x}_{1}},{{y}_{1}} \right\rangle +\left\langle {{T}_{2}}{{x}_{2}},{{y}_{1}} \right\rangle +\left\langle {{T}_{3}}{{x}_{1}},{{y}_{2}} \right\rangle +\left\langle {{T}_{4}}{{x}_{2}},{{y}_{2}} \right\rangle  \right| \\ 
 & \le \left| \left\langle {{T}_{1}}{{x}_{1}},{{y}_{1}} \right\rangle  \right|+\left| \left\langle {{T}_{2}}{{x}_{2}},{{y}_{1}} \right\rangle  \right|+\left| \left\langle {{T}_{3}}{{x}_{1}},{{y}_{2}} \right\rangle  \right|+\left| \left\langle {{T}_{4}}{{x}_{2}},{{y}_{2}} \right\rangle  \right| \\ 
  &\qquad \text{(by the triangle inequality)}\\
 & \le \sqrt{\left\langle {{f}^{2}}\left( \left| {{T}_{1}} \right| \right){{x}_{1}},{{x}_{1}} \right\rangle \left\langle {{g}^{2}}\left( \left| T_{1}^{*} \right| \right){{y}_{1}},{{y}_{1}} \right\rangle }+\sqrt{\left\langle {{f}^{2}}\left( \left| {{T}_{2}} \right| \right){{x}_{2}},{{x}_{2}} \right\rangle \left\langle {{g}^{2}}\left( \left| T_{2}^{*} \right| \right){{y}_{1}},{{y}_{1}} \right\rangle } \\ 
 &\quad +\sqrt{\left\langle {{f}^{2}}\left( \left| {{T}_{3}} \right| \right){{x}_{1}},{{x}_{1}} \right\rangle \left\langle {{g}^{2}}\left( \left| T_{3}^{*} \right| \right){{y}_{2}},{{y}_{2}} \right\rangle }+\sqrt{\left\langle {{f}^{2}}\left( \left| {{T}_{4}} \right| \right){{x}_{2}},{{x}_{2}} \right\rangle \left\langle {{g}^{2}}\left( \left| T_{4}^{*} \right| \right){{y}_{2}},{{y}_{2}} \right\rangle } \\
   &\qquad \text{(by \cite[Theorem 1]{04})}\\  
 & \le \frac{1}{2}\left( \left\langle \left( {{f}^{2}}\left( \left| {{T}_{1}} \right| \right)+{{f}^{2}}\left( \left| {{T}_{3}} \right| \right) \right){{x}_{1}},{{x}_{1}} \right\rangle +\left\langle \left( {{g}^{2}}\left( \left| T_{1}^{*} \right| \right)+{{g}^{2}}\left( \left| T_{2}^{*} \right| \right) \right){{y}_{1}},{{y}_{1}} \right\rangle  \right. \\ 
 &\qquad \left. +\left\langle \left( {{f}^{2}}\left( \left| {{T}_{2}} \right| \right)+{{f}^{2}}\left( \left| {{T}_{4}} \right| \right) \right){{x}_{2}},{{x}_{2}} \right\rangle +\left\langle \left( {{g}^{2}}\left( \left| T_{3}^{*} \right| \right)+{{g}^{2}}\left( \left| T_{4}^{*} \right| \right) \right){{y}_{2}},{{y}_{2}} \right\rangle  \right) \\ 
  &\qquad \text{(by the arithmetic-geometric mean inequality)}.  
\end{aligned}\]
We therefore have
\[\begin{aligned}
   \left\| \left[ \begin{matrix}
   {{T}_{1}} & {{T}_{2}}  \\
   {{T}_{3}} & {{T}_{4}}  \\
\end{matrix} \right] \right\|&\le \frac{1}{2}\left( \left\| {{f}^{2}}\left( \left| {{T}_{1}} \right| \right)+{{f}^{2}}\left( \left| {{T}_{3}} \right| \right) \right\|+\left\| {{g}^{2}}\left( \left| T_{1}^{*} \right| \right)+{{g}^{2}}\left( \left| T_{2}^{*} \right| \right) \right\| \right. \\ 
 &\qquad \left. +\left\| {{f}^{2}}\left( \left| {{T}_{2}} \right| \right)+{{f}^{2}}\left( \left| {{T}_{4}} \right| \right) \right\|+\left\| {{g}^{2}}\left( \left| T_{3}^{*} \right| \right)+{{g}^{2}}\left( \left| T_{4}^{*} \right| \right) \right\| \right). 
\end{aligned}\]
This proves the first desired inequality. The second inequality follows by letting $f(x)=x^{t}$ and $g(x)=x^{1-t}, 0\leq t\leq 1$ in the first inequality.
\end{proof}

We also have the following bound for accretive-dissipative operator matrices.
\begin{theorem}
Let $T\in  \mathcal{B}(\mathcal{H}\oplus\mathcal{H})$ be accretive-dissipative partitioned as in \eqref{05}. Then
\[\|T\|\le \sqrt{{{\left( \left\| {{A}_{11}} \right\|+\left\| {{A}_{22}} \right\| \right)}^{2}}+{{\left( \left\| {{B}_{11}} \right\|+\left\| {{B}_{22}} \right\| \right)}^{2}}}.\]
\end{theorem}
\begin{proof}
It is known that if $X$ is accretive-dissipative operator, then \cite{mir}
\[{{\left\| X \right\|}^{2}}\le {{\left\| \mathcal RX \right\|}^{2}}+{{\left\| \mathcal IX \right\|}^{2}}.\]
Using this fact and the inequality \eqref{09}, we have
\[\begin{aligned}
   \left\| \left[ \begin{matrix}
   {{T}_{11}} & {{T}_{12}}  \\
   {{T}_{21}} & {{T}_{22}}  \\
\end{matrix} \right] \right\|&=
 {{\left\| \left[ \begin{matrix}
   {{A}_{11}} & {{A}_{12}}  \\
   A_{12}^{*} & {{A}_{22}}  \\
\end{matrix} \right]+\textup i\left[ \begin{matrix}
   {{B}_{11}} & {{B}_{12}}  \\
   B_{12}^{*} & {{B}_{22}}  \\
\end{matrix} \right] \right\|}^{2}}\\
 &\qquad \text{(by the second inequality in \eqref{14})}\\ 
 & \le {{\left\| \left[ \begin{matrix}
   {{A}_{11}} & {{A}_{12}}  \\
   A_{12}^{*} & {{A}_{22}}  \\
\end{matrix} \right] \right\|}^{2}}+{{\left\| \left[ \begin{matrix}
   {{B}_{11}} & {{B}_{12}}  \\
   B_{12}^{*} & {{B}_{22}}  \\
\end{matrix} \right] \right\|}^{2}} \\ 
 & \le {{\left( \left\| {{A}_{11}} \right\|+\left\| {{A}_{22}} \right\| \right)}^{2}}+{{\left( \left\| {{B}_{11}} \right\|+\left\| {{B}_{22}} \right\| \right)}^{2}},  
\end{aligned}\]
from which the result follows.
\end{proof}

\section*{Data availability} There is no associated data with this work.
\section*{Competing interest} The authors declare that they have no competing interest.

\vskip 0.3 true cm

\noindent{\tiny (F. Kittaneh) Department of Mathematics, The University of Jordan, Amman, Jordan.

\noindent \textit{E-mail address:} fkitt@ju.edu.jo}

\vskip 0.3 true cm 

\noindent{\tiny (H. R. Moradi) Department of Mathematics, Payame Noor University (PNU), P.O. Box, 19395-4697, Tehran, Iran.
	
\noindent	\textit{E-mail address:} hrmoradi@mshdiau.ac.ir}

\vskip 0.3 true cm 

\noindent{\tiny (M. Sababheh) Vice president, Princess Sumaya University for Technology, Amman, Jordan.}
	
\noindent	{\tiny\textit{E-mail address:} sababheh@psut.edu.jo; sababheh@yahoo.com}

\end{document}